\documentclass[11pt,reqno]{amsart}
\usepackage{amssymb,amsmath
	,mathrsfs,mathtools,titletoc, tikz, stmaryrd}
\usepackage[alphabetic]{amsrefs}
\usepackage{float}
\usepackage[colorlinks, linkcolor=blue,anchorcolor=Periwinkle,
    citecolor=red
]{hyperref}
\usepackage{xcolor}
\usepackage{enumitem}
\usepackage{geometry,array} 
\usepackage{graphicx}
\usepackage{subfigure}
\usepackage{bookmark}
\usepackage{tikz}\usetikzlibrary{matrix}
\usepackage{url}
\usepackage{braids}

\usetikzlibrary{decorations.markings}
\tikzset{->-/.style={decoration={  markings,  mark=at position #1 with
    {\arrow{>}}},postaction={decorate}}}
\tikzset{-<-/.style={decoration={  markings,  mark=at position #1 with
    {\arrow{<}}},postaction={decorate}}}
\usepackage[all]{xy}
\usetikzlibrary{arrows}
\usetikzlibrary{decorations.pathreplacing,decorations.pathmorphing,shadings,fadings,calc}
\usepackage{extarrows}
\usepackage{tikz-3dplot}
\usepackage{pdfpages}

\usepackage{etoolbox}
\patchcmd{\subsection}{-.5em}{.5em}{}{}

\usetikzlibrary{matrix}
\usepackage[all]{xy}
\usetikzlibrary{arrows,shapes}
\usetikzlibrary{calc}
\tikzset{->-/.style={decoration={  markings,  mark=at position #1 with
    {\arrow{>}}},postaction={decorate}}}
\tikzset{-<-/.style={decoration={  markings,  mark=at position #1 with
    {\arrow{<}}},postaction={decorate}}}
\usepackage[]{hyperref}
\numberwithin{equation}{section}

\usetikzlibrary{decorations.markings}
\tikzset{->-/.style={decoration={  markings,  mark=at position #1 with
    {\arrow{>}}},postaction={decorate}}}
\tikzset{-<-/.style={decoration={  markings,  mark=at position #1 with
    {\arrow{<}}},postaction={decorate}}}

\usepackage{lscape}
\usepackage{CJK}

\newtheorem{theorem}{Theorem}[section]
\newtheorem{prop}[theorem]{Proposition}
\newtheorem{lemma}[theorem]{Lemma}

\newtheorem{definition}[theorem]{Definition}
\newtheorem{cor}[theorem]{Corollary}

\numberwithin{equation}{section}

\begin{document}
\pdfoutput=1 
\def\nn{node{$\bullet$}}
\def\ww{node{$\circ$}}

\title[Classification of Willmore surfaces with vanishing Gaussian curvature]{Classification of Willmore surfaces with vanishing Gaussian curvature}
\author{Yunqing Wu}
\address{Department of Mathematical Sciences\\Tsinghua University, Beijing\\P.R.China, 100084}
\email{yq-wu19@mails.tsinghua.edu.cn}

\begin{abstract}
	We classify simply-connected, complete Willmore surfaces with vanishing Gaussian curvature. We also study the Willmore cones in $\mathbb{R}^{3}$ and give a classification. As an application, we show that for a complete Willmore  embedding $f:\mathbb{R}^{2} \rightarrow \mathbb{R}^{3}$, if its corresponding Gaussian curvature is nonnegative and the image of its Gauss map lies in $\{ \theta \in \mathbb{S}^{2}: d_{\mathbb{S}^{2}}(\theta,\theta_{0}) \leq \alpha<\frac{\pi}{2}  \}$,  $f(\mathbb{R}^{2})$ is a plane. 
\end{abstract}

\maketitle

\titlecontents{section}[0em]{}{\hspace{.5em}}{}{\titlerule*[1pc]{.}\contentspage}
\titlecontents{subsection}[1.5em]{}{\hspace{.5em}}{}{\titlerule*[1pc]{.}\contentspage}
\tableofcontents


\section{Introduction}
Let $f:\mathbb{R}^{2} \rightarrow \mathbb{R}^{3}$ be an embedding, its Willmore functional is defined by
\begin{align}
\label{Intro1}
\mathcal{W}(f)=\frac{1}{4} \int_{ \mathbb{R}^{2} } |H|^{2} d \mu_{g},
\end{align}
where $g$ is the induced metric, $H=\Delta_{g} f$ is the mean curvature. Let $K_{g}$ be the Gaussian curvature. We say that $f$ is Willmore if it is the critical point of $\mathcal{W}$ and it is the solution of the Euler-Lagrange equation:
\begin{align}
\label{Intro2}
\Delta_{g} H+\frac{1}{2} (|H|^{2}-4K_{g}  )H=0. 
\end{align}

Willmore surfaces seem to be more complicated to study in contrast with minimal surfaces in view of the equation since (\ref{Intro2}) is a fourth order partial differential equation, which is hard to deal with directly. 
A natural question is to ask whether some well-known results of minimal surfaces still hold on general Willmore surfaces or not. For example, the classic Bernstein theorem says that every entire minimal graph in $\mathbb{R}^{3}$ is a plane. However, it is still unknown whether an entire Willmore graph satisfies Bernstein type theorem. 

There are some  Bernstein type theorems of entire Willmore graphs and Willmore surfaces under some extra assumptions. In \cite{CL131}, Chen and Lamm proved that every entire Willmore graph with finite $L^{2}$ norm of the second fundamental form is a plane. They also proved any entire Willmore graph with the form $\{(x,y,u(x)): (x,y) \in \mathbb{R}^{2}\} $ is a plane. 
Luo and Sun further proved that every entire Willmore graph with finite Willmore functional is still a plane in \cite{LS14}. In \cite{Li2016}, Li proved that for any complete Willmore embedding $f:\mathbb{R}^{2} \rightarrow \mathbb{R}^{3}$, $f(\mathbb{R}^{2})$ is a plane if it has finite $L^{2}$ norm of the second fundamental form. However, we should note that a complete Willmore embedding with finite Willmore functional may not be a plane since a helicoid is a counterexample. In \cite{ChenLi2017}, Chen and Li proved that any entire radially symmetric Willmore graph is a plane. 

Inspired by the above known results, we study whether a complete Willmore surface is a plane provided the $L^{1}$ norm of its Gaussian curvature is finite. Unfortunately, we can prove that the Bernstein type theorem fails on such Willmore surfaces by classifying simply-connected, complete Willmore surfaces with vanishing Gaussian curvature zero. 

Applying a classification result of complete surfaces with vanishing Gaussian curvature by Hartman and Nirenberg\cite{Hartman1959OnSI}(also by Massey\cite{Massey1962SurfacesOG}), a complete Willmore surface of Gaussian curvature zero can be locally written as a graph $\{(x,y,u(x)): x \in \mathcal{I} \subset \mathbb{R}, y \in \mathbb{R} \}$ if we choose an appropriate coordinate system. Therefore, (\ref{Intro2}) reduces to the following equation
\begin{align}
\label{Intro 3}
(      \frac{     w^{\prime \prime} ( 1+w^{2}) -\frac{5}{2} w( w^{\prime}  )^{2}          }{    (1+w^{2})^{\frac{7}{2}}              }               )^{\prime}=0,
\end{align}
where $w=u^{\prime}$. There exists a unique local solution $w$ to (\ref{Intro 3}) endowed with the initial conditions: 
$
w(0)=0, w^{\prime}(0)=0, w^{\prime \prime}(0)=C>0. 
$
The initial conditions imply geometric information: $\displaystyle H|_{(0,y,u(0))}=0, |\nabla_{\Sigma_{C}} H|_{(0,y,u(0))}=C$, where $\Sigma_{C}$ is the corresponding graph. By the simply-connectness assumption, we can show that adding such initial conditions is always allowed. By analysing the symmetry of $w$ and using the result in \cite{CL131}, we can extend $\Sigma_{C}$ and obtain a smooth periodic Willmore surface (still denoted by $\Sigma_{C}$). Nevertheless, $\Sigma_{C}$ is only an entire continuous Willmore graph.  In conclusion, we state the following theorem which gives a classification of Willmore surfaces with vanishing Gaussian curvature.  
 
\begin{theorem}
	\label{theorem intro}
	(1). 
	There exists a family of non-flat Willmore surfaces $\{ \Sigma_{C}:C \in \mathbb{R}^{+} \}$ satisfying: $\Sigma_{C}$ is determined only by $C$, $\Sigma_{C}$ has zero Gaussian curvature and the image of $\Sigma_{C}$'s Gauss map is  a great semi-circle. \\
	(2). Given any simply-connected, complete, non-flat Willmore surface $\Sigma$ with vanishing Gaussian curvature, there exists $C>0$ and an isometric transformation $\mathscr{I}_{C}$ in $\mathbb{R}^{3}$ such that
	$\Sigma=\mathscr{I}_{C}(\Sigma_{C})$. 
	\end{theorem}

Theorem \ref{theorem intro} also implies that Willmore surfaces do not satisfy the Osserman type theorem. 
Recall that Osserman conjectured that the Gauss map of an orientable, complete, non-flat, immersed minimal surface with finite total curvature in $\mathbb{R}^{3}$ cannot miss 3 points of $\mathbb{S}^{2}$. In \cite{Fujimoto}, Fujimoto proved that the Gauss map of any orientable, complete, non-flat minimal surface in $\mathbb{R}^{3}$ can omit at most four points of $\mathbb{S}^{2}$. Fujimoto's result is optimal since the Gauss map of a Scherk surface omits four points. However, the Gauss map of a Willmore surface could miss almost all points of $\mathbb{S}^{2}$ by Theorem \ref{theorem intro} and this is a geometric difference between  minimal surfaces and general Willmore surfaces. 

We next study the Willmore cones in $\mathbb{R}^{3}$. More precisely, we consider $\mathscr{C}_{\gamma}=\{r \gamma(\mathbb{R}):r>0\}$ where $\gamma:\mathbb{R} \rightarrow \mathbb{S}^{2}$ is an immersion with $|\gamma(s)|=1, |\gamma^{\prime}(s)|=1$. We call it a Willmore cone if it satisfies the Willmore equation (\ref{Intro2}). Let $H_{\mathscr{C}_{\gamma} }$ and $K_{\mathscr{C}_{\gamma} }$ denote the mean curvature and Gaussian curvature of $\mathscr{C}_{\gamma}$, respectively. Due to the cone structure, we could obtain $K_{\mathscr{C}_{\gamma} }=0$, $\displaystyle H_{\mathscr{C}_{\gamma} }(r,s)=\frac{ \mathcal{H}(s)}{r}$ where $\mathcal{H}$ satisifies
$\displaystyle
\mathcal{H}^{\prime \prime}(s)=-\mathcal{H}(s)( 1+\frac{1}{2} \mathcal{H}^{2}(s)   ). 
$
It is easy to check that $\mathcal{H}^{\prime}$ must has a zero point if it is an entire function. We set $\mathcal{H}_{a}$ be a solution of 
\begin{align*}
\mathcal{H}^{\prime \prime}(s)=-\mathcal{H}(s)( 1+\frac{1}{2} \mathcal{H}^{2}(s)   ), \quad \mathcal{H}(0)=a, \quad \mathcal{H}^{\prime}(0)=a,
\end{align*}
and $\gamma_{a}(s)$ be a curve whose intrinsic curvature in $\mathbb{S}^{2}$ is $\mathcal{H}_{a}(s)$. 
We will show that $\mathcal{H}_{a}$ and $\gamma_{a}$ are periodic. Using the symmetry and periodic property of $\mathcal{H}_{a}$, we can show that if the image of some $\gamma$ is embedded and closed in $\mathbb{S}^{2}$(which implies $\mathscr{C}_{\gamma}$ has at most one singularity), its length will be strictly larger than 2$\pi$ unless it is not a great circle. 

As an application, we can prove the following  theorem. 
\begin{theorem}
	\label{theorem berstein-type main theorem}
	There does not exist a complete  Willmore  embedding $f:\mathbb{R}^{2} \rightarrow \mathbb{R}^{3}$ satisfying the following properties: \\
	($\mathscr{M}1$). There exists $C>0$ such that $0 < \displaystyle \int_{\mathbb{R}^{2}} K_{g} d \mu_{g} \leq 2\pi, \quad \int_{\mathbb{R}^{2}} |K_{g}| d \mu_{g} \leq C $. \\
	($\mathscr{M}2$). The image of the corresponding Gauss map lies in $\{ \theta \in \mathbb{S}^{2}: d_{\mathbb{S}^{2}}(\theta,\theta_{0}) \leq \alpha<\frac{\pi}{2}  \}$.
\end{theorem}

The strategy of the proof is as follows.  By contradiction, we assume there exists such an embedding $f:\mathbb{R}^{2} \rightarrow \mathbb{R}^{3}$, then an intrinsic blow-down sequence of $f$ will converge to a limit $f_{\infty}$ smoothly on $\mathbb{R}^{2} \setminus \{0\}$. With the help of the results in \cite{Hartman1959OnSI}, we can analyse the image of $f_{\infty}$ . Using Theorem \ref{theorem intro} and the properties of  Willmore cones, we can conclude that $\displaystyle \int_{\mathbb{R}^{2}} K_{g} d \mu_{g}=0$, which leads to a contradiction. 

The following Bernstein-type theorem follows from Cohn-Vossen theorem and Theorem \ref{theorem berstein-type main theorem} immediately. 

\begin{theorem}
	\label{theorem berstein-type 2}
	Let $f: \mathbb{R}^{2} \rightarrow \mathbb{R}^{3}$ be a complete  Willmore embedding. If $K_{g} \geq 0$ and the image of the corresponding Gauss map lies in $\{ \theta \in \mathbb{S}^{2}: d_{\mathbb{S}^{2}}(\theta,\theta_{0}) \leq \alpha<\frac{\pi}{2}  \}$, $f(\mathbb{R}^{2})$ is a plane. 
\end{theorem}
\begin{proof}
	Since $K_{g} \geq 0$, by Cohn-Vossen theorem, $\displaystyle \int_{\mathbb{R}^{2}} K_{g} d \mu_{g} \leq 2 \pi$. By Theorem \ref{theorem berstein-type main theorem}, $\displaystyle \int_{\mathbb{R}^{2}} K_{g} d \mu_{g} =0$, which implies $K_{g} \equiv 0$. By Theorem \ref{theorem intro} and the assumption, we conclude that $f(\mathbb{R}^{2})$ is a plane. 
\end{proof}

The organization of this article is as follows. In section 2, we will study Willmore surfaces with vanishing Gaussian curvature and  prove Theorem \ref{theorem intro}. 
In section 3, we will analyse and classify the Willmore cones. In section 4, we will prove Theorem \ref{theorem berstein-type main theorem}. Appendix contains a key lemma in \cite{Hartman1959OnSI}, we give the proof of a 2-dimensional case. 

\section{Nontrivial Willmore surfaces with vanishing Gaussian curvature}
Throughout this section, we assume $\Sigma$ is a simply-connected, complete Willmore surface with vanishing Gaussian curvature. 

The following theorem, which is proved in \cite{Hartman1959OnSI} and \cite{Massey1962SurfacesOG}, shows a splitting property of a complete surface with vanishing  Gaussian curvature in $\mathbb{R}^{3}$. 

\begin{theorem}	
	\label{theorem complete surface in R3 is a cylinder}
	A complete surface of Gaussian curvature zero in $\mathbb{R}^{3}$ is a``cylinder". More precisely, up to a  rotation, a ``cylinder" means the surface is generated by the set of lines parallel to the $y$-axis through a curve in the $xz$-plane.  
\end{theorem}

By Theorem \ref*{theorem complete surface in R3 is a cylinder}, $\Sigma$ can be locally written as 
$\displaystyle 
\{ (x,y,u(x)): x \in \mathcal{I}\subset \mathbb{R},  y \in \mathbb{R} \}.
$
From \cite{DD06}, 
$u$ satisfies the following Willmore graphic equation
\begin{align}
\label{Graphic Willmore surface equation}
div \big(   \frac{1}{v} \big( \big(Id-\frac{\nabla u \otimes \nabla u}{v^{2}}    \big)   \nabla (vH)  -\frac{1}{2} H^{2} \nabla u     \big)                \big)=0,
\end{align}
where $\nabla$ is the Euclidean gradient operator, $\displaystyle v=\sqrt{1+|\nabla u|^{2}}$ and $\displaystyle  H=div(\frac{\nabla u}{v} )$.

By standard calculations, we could obtain the following ODE equation. 
\begin{prop}
	\label{prop ode of the solution}
	If $u=u(x)$ solves the graphic Willmore surface equation  (\ref{Graphic Willmore surface equation}), then 
\begin{align}
\label{1-dim Graphic Willmore surface equation}
	\frac{d}{dx}(      \frac{     w^{\prime \prime} ( 1+w^{2}) -\frac{5}{2} w( w^{\prime}  )^{2}          }{    (1+w^{2})^{\frac{7}{2}}              }               )=0,
	\end{align}
	where $w=u^{\prime}$. 
\end{prop}

\begin{proof}
	Clearly, $v$ and $H$ are functions of $x$ since $u$ depends only on $x$. 
	By direct calculations, 
	\begin{align}
	\frac{dv}{dx}&=\frac{ u^{\prime} u^{\prime \prime}}{  \sqrt{1+(u^{\prime})^{2}}    }=\frac{ w w^{\prime}}{  \sqrt{1+w^{2}}  } \label{v'}, \\
	H&=div(  \frac{\nabla u}{v}     ) =\frac{d}{dx}\big( \frac{ u^{\prime}}{ \sqrt{ 1+(u^{\prime})^{2}  }  }     \big) \nonumber \\&=\frac{  u^{\prime \prime}  \sqrt{ 1+(u^{\prime})^{2}  } -  u^{\prime} \frac{u^{\prime} u^{\prime \prime}  }{ \sqrt{ 1+(u^{\prime})^{2}  }   }    }{  1+(u^{\prime})^{2}  }    =\frac{ w^{\prime}}{ (1+w^{2})^{\frac{3}{2}}    }, \label{H}\\
	\frac{dH}{dx}&=\frac{  w^{\prime \prime}  (1+w^{2})^{\frac{3}{2}} -w^{\prime} \cdot \frac{3}{2} (1+w^{2})^{\frac{1}{2}} \cdot 2 w w^{\prime}     }{ (1+w^{2})^{3}   } \nonumber\\&=\frac{ w^{\prime \prime} (1+w^{2})-3(w^{\prime})^{2} w}{ (1+w^{2})^{\frac{5}{2}}       } \label{H'}.
\end{align}

Then
\begin{align*}
&(Id-\frac{\nabla u \otimes \nabla u}{v^{2}}    )   \nabla (vH)  -\frac{1}{2} H^{2} \nabla u   \\
=& (  1-\frac{ (u^{\prime})^{2}}{v^{2}} ) ( vH)_{x}  \frac{\partial}{\partial x}   +(  vH         )_{y} \frac{\partial}{\partial y} -\frac{1}{2} H^{2} u^{\prime} \frac{\partial}{\partial x} \\
=& \big\{\frac{1}{1+w^{2}} \big[ \frac{ w w^{\prime}}{  \sqrt{1+w^{2}}  } \cdot    \frac{ w^{\prime}}{ (1+w^{2})^{\frac{3}{2}}    } +\sqrt{1+w^{2}} \cdot  \frac{ w^{\prime \prime} (1+w^{2})-3(w^{\prime})^{2} w}{ (1+w^{2})^{\frac{5}{2}}       }            \big] \\
&-\frac{1}{2} (  \frac{ w^{\prime}}{ (1+w^{2})^{\frac{3}{2}}    }         )^{2} w \big\} \frac{\partial}{\partial x} \\
=& \frac{     w^{\prime \prime} ( 1+w^{2}) -\frac{5}{2} w( w^{\prime}  )^{2}          }{    (1+w^{2})^{3}              } \frac{\partial}{\partial x} .
\end{align*}
Combining all the above, we have
\begin{align*}
	&div \big(   \frac{1}{v} \big( \big(Id-\frac{\nabla u \otimes \nabla u}{v^{2}}    \big)   \nabla (vH)  -\frac{1}{2} H^{2} \nabla u     \big)                \big)\\
	=& \frac{d}{dx}(      \frac{     w^{\prime \prime} ( 1+w^{2}) -\frac{5}{2} w( w^{\prime}  )^{2}          }{    (1+w^{2})^{\frac{7}{2}}              }               )=0.
\end{align*}

By integration, there exists a  constant C such that
\begin{align}
\label{ODE 2}
w^{\prime \prime} ( 1+w^{2}) -\frac{5}{2} w( w^{\prime}  )^{2}      =C  (1+w^{2})^{\frac{7}{2}}       .
\end{align}

Observing that
$$
\frac{d}{dx} [( 1+w^{2}    )^{-\frac{5}{4}} w^{\prime}  ]=(1+w^{2})^{-\frac{9}{4}} [    (1+w^{2}) w^{\prime \prime} -\frac{5}{2} w (w^{\prime})^{2}             ],
$$
we can also write (\ref{ODE 2}) as
\begin{align}
\label{ODE 3}
\frac{d}{dx} [( 1+w^{2}    )^{-\frac{5}{4}} w^{\prime}  ]=C (1+w^{2})^{\frac{5}{4}} .
\end{align}

\end{proof}

	The following Bernstein type assertion has been proved in \cite[Appendix]{CL131} 
	\begin{prop}
		\label{prop no entire solution}
		If the corresponding solution $w=w(C)$ of (\ref{ODE 2}) is an entire solution, $w$ is constant. 
	\end{prop}
    
Now we come to prove Theorem \ref{theorem intro}. 
 \begin{proof}
 	\textbf{Step 1:} 
 		Since $\Sigma$ is simply-connected and $K=0$, $\Sigma$ is isometric to $\mathbb{R}^{2}$. Therefore, we can view $H$ as a function from $\mathbb{R}^{2}$ to $\mathbb{R}$ which satisfies $\displaystyle \Delta_{\mathbb{R}^{2}} H=-\frac{1}{2} H^{3}$. 
 		
 		If $H(p)>0$ for $\forall p \in \mathbb{R}^{2}$, $-H$ is a subharmonic function which is bounded above, hence is constant by the Liouville theorem, which leads to $\Delta_{\mathbb{R}^{2}} H \equiv 0$. Then $H \equiv 0$ and we obtain a contradiction. 
 		If $H(p)<0$ for $ \forall p \in \mathbb{R}^{2}$, then $H$ is a subharmonic function which is bounded above, hence is constant again  which leads to $H \equiv 0$, a contradiction to  the equation assumption. We conclude that there exists $p \in \Sigma$ such that $H(p)=0$. 
 		
 	Therefore, we can write $\Sigma$ locally  as 
 	$
 	\{ (x,y,u(x)): x \in \mathcal{I}, y \in \mathbb{R}  \}
 	$ 
 	under some appropriate coordinate system such that: $w=u^{\prime}$ satisfies (\ref{ODE 2}), 
 	the projection of $p$ on the $xz$-plane is $(0,u(0))$, $u(0)=0$, $u^{\prime}(0)=w(0)=0$, where $\mathcal{I}$ is the maximal existence interval of (\ref{ODE 2}). By (\ref{H}), we also have $w^{\prime}(0)=0$. If $C=0$, we will obtain $u \equiv 0$ and $\Sigma $ is a plane. 
 	
 		\textbf{Step 2:} 
 	Now we assume $C>0$.  By (\ref{H'}), 
 		\begin{align}
 	\label{NE1}
 	|\nabla_{\Sigma} H|^{2}|_{( x,y,u(x)   )    }=[(1+w^{2})^{-1}( \frac{ w^{\prime \prime} (1+w^{2})-3(w^{\prime})^{2} w}{ (1+w^{2})^{\frac{5}{2}}       }     )^{2}](x), \quad \forall x \in \mathcal{I}.
 		\end{align}
 	In particular, $|\nabla_{\Sigma} H|^{2}( p     )=(w^{\prime \prime}(0))^{2}=C^{2}$.
 	
 	By the uniqueness of the solution and (\ref{ODE 2}), $w$ is even and $u$ is odd on $\mathcal{I}$. By Proposition \ref{prop no entire solution}, we know that $\mathcal{I}=(-\rho,\rho) \text{ or } [-\rho, \rho]$, where $\rho \in (0,+\infty)$. If $\mathcal{I}=(-\rho,\rho)$, then,  
 	$\displaystyle 
 	\lim_{x \rightarrow \rho-0} u(x) = +\infty, \lim_{x \rightarrow -\rho+0} u(x) = -\infty. 
  	$
  	After taking a coordinate transformation
  	$$
  	\hat{x}=x \cos \alpha  +z \sin \alpha, \quad \hat{z} =-x \sin \alpha+z \cos \alpha, 
  	$$
  	where $0<\frac{\pi}{2}-\alpha$ is sufficiently small, we will obtain a nontrivial entire solution of (\ref{ODE 2}) in new coordinates, a contradiction to Proposition \ref{prop no entire solution}. Therefore, we obtain
  	$$
  	u(\rho )=\lim_{x \rightarrow \rho-0} u(x) < +\infty, \quad u(-\rho )=\lim_{x \rightarrow -\rho+0} u(x) > -\infty, \quad \mathcal{I}=[-\rho,\rho].
  	$$
  	
  	Set $\displaystyle \xi=u(\rho) $. 	Integrate both sides of (\ref{ODE 3}), 
  	\begin{align}
  	\label{NE2}
  	w^{\prime}(x)=C ( 1+w^{2} (x)   )^{\frac{5}{4}} \int_{0}^{x} (1+w^{2}(t) )^{\frac{5}{4}} dt, \quad \forall x \in \mathcal{I}.
  	\end{align}
  	
  	By (\ref{NE2}), $\displaystyle \lim_{x \rightarrow \rho-0} w^{\prime}(x)=+\infty$ if and only if $\displaystyle \lim_{x \rightarrow \rho-0} w(x)=+\infty$. By the maximality of $\mathcal{I}$, $\displaystyle \lim_{x \rightarrow \rho-0} w^{\prime}(x)=+\infty.$ 
  	By (\ref{NE2}) again, 
  	\begin{align}
  	\label{NE3}
  	w^{\prime}(x)=\frac{C}{2}( [   \int_{0}^{x} (1+w^{2}(t) )^{\frac{5}{4}} dt             ]^{2} )^{\prime}, 
\end{align}  	
then 
\begin{align}
\label{NE4}
w(x)=\frac{C}{2} [   \int_{0}^{x} (1+w^{2}(t) )^{\frac{5}{4}} dt             ]^{2} .
\end{align}  	
By (\ref{H}), (\ref{NE2}) and (\ref{NE4}), for $x>0$, 
\begin{align}
H(x)&=\frac{ w^{\prime}(x)}{     (1+w^{2}(x))^{\frac{3}{2}}          }\nonumber =\frac{ C ( 1+w^{2} (x)   )^{\frac{5}{4}} \int_{0}^{x} (1+w^{2}(t) )^{\frac{5}{4}} dt}{   (1+w^{2}(x))^{\frac{3}{2}}         } \nonumber\\
&=\frac{ C \int_{0}^{x} (1+w^{2}(t) )^{\frac{5}{4}} dt}{  \sqrt{w(x)} (   1+\frac{1}{w^{2}(x)}       )^{\frac{1}{4}}      }=\frac{\sqrt{2C} }{ (1+\frac{1}{w^{2}(x)}       )^{\frac{1}{4}}           }. \label{NE5}
\end{align}
In particular, 
$\displaystyle 
H(\rho)=\lim_{x \rightarrow \rho} H(x)=\sqrt{2C}>0, \quad H(-\rho)=-H(\rho)=-\sqrt{2C}. 
$

By (\ref{ODE 2}), (\ref{NE1}) and (\ref{NE5}), 
\begin{align}
|\nabla_{\Sigma} H|^{2}((x,y,u(x) )) =& \frac{  [C(1+w^{2}(x))^{\frac{7}{2}} -\frac{1}{2} ( w^{\prime}(x)   )^{2} w(x)     ]^{2}   }{  (1+w^{2}(x))^{6}     } \nonumber \\= & \frac{[C(1+w^{2}(x))^{\frac{7}{2}} -\frac{1}{2} w(x) (      \frac{2C(1+w(x))^{3}}{  (1+\frac{1}{w^{2}(x)}       )^{\frac{1}{2}}          }        )]^{2}   }{  (1+w^{2}(x))^{6}     } =\frac{C^{2}}{ 1+w^{2}(x)   }. \label{NE6}
\end{align}

In particular, 
$\displaystyle 
|\nabla_{\Sigma} H|^{2}((\rho,y,u(\rho) )) =|\nabla_{\Sigma} H|^{2}((-\rho,y,u(-\rho) )) =\lim_{|x|\rightarrow \rho} \frac{C^{2}}{ 1+w^{2}(x)   }=0. 
$

	\textbf{Step 3:} 
Now we consider a new coordinate system
$\displaystyle 
\tilde{x}=z+\xi, \quad \tilde{z}=-x+\rho. 
$
Again, $\Sigma$ can be locally written as $\{ \tilde{x}, y, \tilde{u}( \tilde{x}  )  :\tilde{x} \in \tilde{\mathcal{I}}, y \in \mathbb{R}  \}$, where 
$\displaystyle 
\tilde{u}(0)=0, \quad \tilde{u}^{\prime}(0)=\lim_{x \rightarrow \rho} \frac{1}{ w(x)  }=0, 
$
and $\tilde{w}=\tilde{u}^{\prime}$ satisfies (\ref{ODE 3}) for some $\tilde{C}$. In this new coordinate system, 
$\displaystyle 
H((0,y,\tilde{u}(0)))=\sqrt{2C}, \quad |\nabla_{\Sigma}H|^{2}(  (0,y,\tilde{u}(0))   )=0, 
$
which implies $\tilde{C}=0$ and 
\begin{align*}
&\tilde{w}^{\prime \prime}( 1+ \tilde{w}^{2}  )=\frac{5}{2} \tilde{w} (  \tilde{w}^{\prime}   )^{2}, \quad \forall x \in \tilde{ \mathcal{I}}, \\
&\tilde{w}(0)=0, \quad \tilde{w}^{\prime}(0)=\sqrt{2C}. 
\end{align*}
By the uniqueness of solution, $\tilde{w}$ is odd, $\tilde{u}$ is even and $\tilde{I}=[-\xi,\xi]$. 

\begin{figure}[H]
	\centering
	\includegraphics[scale=0.4]{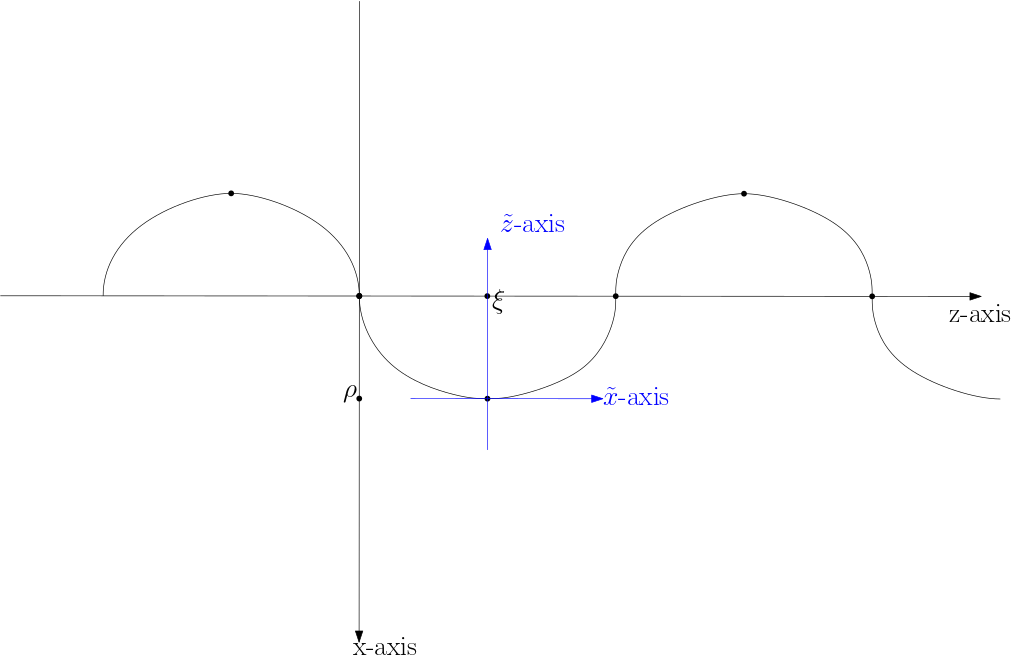}
	\caption{}
	\label{Figure 1}
\end{figure}

Repeating the above argument, we obtain a periodic Willmore surface with vanishing Gaussian curvature. Its projection on the $xz$-plane is shown in  Figure \ref{Figure 1}.

In conclusion, for any $C>0$, we can find a Willmore surface $\Sigma_{C}$ with vanishing Gaussian curvature. Clearly, there is no difference if we assume $C<0$ and $\Sigma_{-C}$ is isometric to $\Sigma_{C}$ for any $C>0$. On the other hand, given any simply-connected, complete Willmore surface $\Sigma$, $\Sigma$ is isometric to $\Sigma_{C}$ for some $C>0$. 

For any $C>0$, the image of $\Sigma_{C}$'s Gauss map lies in a great semi-circle of $\mathbb{S}^{2}$ (with Hausdorff dimension 1). 
\end{proof}

\section{Willmore cone}

In this section, we consider the Willmore cones defined as follows. 

\begin{definition}
	\label{defn Willmore cone}
	Assume $\gamma:\mathbb{R} \rightarrow \mathbb{S}^{2}$ is an immersion. Denote
	\begin{align*}
	\mathscr{C}_{\gamma}:=\{ r \gamma(s): r>0, |\gamma(s)|=1, |\gamma^{\prime}(s)|=1, s \in \mathbb{R} \},
	\end{align*}
	we say $\mathscr{C}_{\gamma}$ is a Willmore cone if 
	\begin{align*}
	\Delta_{\mathscr{C}_{\gamma} } H_{\mathscr{C}_{\gamma}}+\frac{1}{2} H_{\mathscr{C}_{\gamma}}^{3}-2H_{\mathscr{C}_{\gamma}} K_{\mathscr{C}_{\gamma}}=0,
	\end{align*}
	where $H_{\mathscr{C}_{\gamma}}$ and $K_{\mathscr{C}_{\gamma}}$ denote the mean curvature and Gaussian curvature of $\mathscr{C}_{\gamma}$, respectively.
	
	Clearly, $\mathscr{C}_{\gamma}$ is a Willmore cone if we take $\gamma$ a great circle of  $\mathbb{S}^{2}$. 
\end{definition}

\begin{prop}
	\label{prop Willmore cone ode}
	For a Willmore cone $\mathscr{C}_{\gamma}$, set $\displaystyle \mathcal{H}(s)=\gamma^{\prime \prime}(s) \cdot (\gamma(s) \times \gamma^{\prime}(s) ) $. Then
	\begin{align*}
	H_{\mathscr{C}_{\gamma}}( r\gamma(s) )=\frac{ \mathcal{H}(s)}{r}, \quad K_{\mathscr{C}_{\gamma}} \equiv 0, \quad \mathcal{H}^{\prime \prime}(s)=-\mathcal{H}(s)( 1+\frac{1}{2} \mathcal{H}^{2}(s)   ). 
	\end{align*}
\end{prop}

\begin{proof}
For simplicity, we set
\begin{align*}
f(r,s)&=r\gamma(s), \quad \mathcal{E}=f_{r} \cdot f_{r}, \quad \mathcal{F}=f_{r} \cdot f_{s}, \quad \mathcal{G}=f_{s} \cdot f_{s}, \\
n&=\frac{ f_{r} \times f_{s}}{ | f_{r} \times  f_{s}| }, \quad \mathcal{L}=f_{rr} \cdot n, \quad \mathcal{M}=f_{rs} \cdot n, \quad \mathcal{N}=f_{ss} \cdot n.
\end{align*}

By direct calculations, we obtain
\begin{align*}
f_{r}&=\gamma(s), \quad f_{s}=r \gamma^{\prime}(s), \\
\mathcal{E}&=|\gamma(s)|^{2}=1, \quad \mathcal{F}=r\gamma(s) \cdot \gamma^{\prime}(s)=0,\quad \mathcal{G}=r^{2}|\gamma^{\prime}(s)|^{2}=r^{2},\\
n&=\frac{ r \gamma(s) \times \gamma^{\prime}(s) }{ | r \gamma(s) \times \gamma^{\prime}(s)|   }=\gamma(s) \times \gamma^{\prime}(s) ,\\
\mathcal{L}&=0,\quad \mathcal{M}=\gamma^{\prime}(s) \cdot (\gamma(s) \times \gamma^{\prime}(s)) =0, \quad \mathcal{N}= r\gamma^{\prime \prime}(s) \cdot (\gamma(s) \times \gamma^{\prime}(s) ), 
\end{align*}
then 
\begin{align}
\label{Willmore cone H}
K_{\mathscr{C}_{\gamma}}( r\gamma(s) )&=\frac{ \mathcal{L}\mathcal{N}-\mathcal{M}^{2}}{\mathcal{E}\mathcal{G}-\mathcal{F}^{2}}=0, \\H_{\mathscr{C}_{\gamma}}( r\gamma(s) )&= \frac{ \mathcal{L}\mathcal{G}-2\mathcal{M}\mathcal{F}+\mathcal{N}\mathcal{E}}{\mathcal{E}\mathcal{G}-\mathcal{F}^{2}   }=\frac{ \gamma^{\prime \prime}(s) \cdot (\gamma(s) \times \gamma^{\prime}(s) )  }{r}=\frac{\mathcal{H}(s)}{r}. 
\end{align}
Since  $\mathscr{C}_{\gamma}$ is a Willmore cone, 
\begin{align*}
0&=\Delta_{\mathscr{C}_{\gamma} } H_{\mathscr{C}_{\gamma}}+\frac{1}{2} H_{\mathscr{C}_{\gamma}}^{3}-2H_{\mathscr{C}_{\gamma}} K_{\mathscr{C}_{\gamma}}\\
&=\frac{1}{r} \{   \frac{\partial}{\partial r}(  r\frac{\partial}{\partial r} ( \frac{\mathcal{H}(s)}{r}  )     )  +\frac{\partial}{\partial s}( \frac{r}{r^{2}} \frac{\partial}{\partial s} ( \frac{\mathcal{H}(s)}{r}  )      )               \}  +\frac{1}{2}  \frac{\mathcal{H}^{3}(s)}{r^{3}}  \\
&=\frac{\mathcal{H}(s)}{r^{3}}+\frac{\mathcal{H}^{\prime \prime}(s)}{r^{3} }+\frac{1}{2}  \frac{\mathcal{H}^{3}(s)}{r^{3}}  ,
\end{align*}
which implies
\begin{align}
\label{Willmore cone equation}
\mathcal{H}^{\prime \prime}(s)=-\mathcal{H}(s)( 1+\frac{1}{2} \mathcal{H}^{2}   ). 
\end{align}
\end{proof}

	Indeed, by Gauss equation, 
	$\displaystyle 
	\nabla^{\mathbb{S}^{2}}_{\gamma^{\prime}} \gamma^{\prime}=\gamma^{\prime \prime}-(\gamma^{\prime \prime} \cdot \gamma ) \gamma. 
	$
	Note that 
	\begin{align*}
	&\nabla^{\mathbb{S}^{2}}_{\gamma^{\prime}} \gamma^{\prime} \cdot \gamma=\gamma^{\prime \prime} \cdot \gamma-\gamma^{\prime \prime} \cdot \gamma =0,\\
	& 	\nabla^{\mathbb{S}^{2}}_{\gamma^{\prime}} \gamma^{\prime} \cdot \gamma^{\prime}=(\gamma^{\prime \prime}-(\gamma^{\prime \prime} \cdot \gamma ) \gamma) \cdot \gamma^{\prime}=0, \\&
	\nabla^{\mathbb{S}^{2}}_{\gamma^{\prime}} \gamma^{\prime} \cdot (\gamma \times \gamma^{\prime})=\gamma^{\prime \prime} \cdot  (\gamma \times \gamma^{\prime}), 
	\end{align*}
	which implies an intrinsic ODE equation $\displaystyle \nabla^{\mathbb{S}^{2}}_{\gamma^{\prime}} \gamma^{\prime}=\mathcal{H} \mathcal{J}(\gamma^{\prime})$, where $\mathcal{J}$ is a rotation operator. More precisely, assume $\{E_{1},E_{2}\}$ is a orthonormal frame on $\mathbb{S}^{2}$, then $\mathcal{J}(E_{1}(p))=E_{2}(p), \mathcal{J}(E_{2}(p))=-E_{1}(p)$. By the standard ODE theorey, given a function $\mathcal{H}$, there exists a unique $\gamma \subset \mathbb{S}^{2}$ solving the equation $\displaystyle \nabla^{\mathbb{S}^{2}}_{\gamma^{\prime}} \gamma^{\prime}=\mathcal{H} \mathcal{J}(\gamma^{\prime})$ with initial data: $\gamma(0)=p, \gamma^{\prime}(0)=v$. 

Now we come to analyse the equation (\ref{Willmore cone equation}) with initial conditions as follows. 
\begin{prop}
	\label{prop Willmore cone ode comparison}
	Assume $\mathcal{H}_{a}$ satisfies
	\begin{align*}
	\mathcal{H}_{a}^{\prime \prime}&=-\mathcal{H}_{a}( 1+\frac{1}{2} \mathcal{H}_{a}^{2}  ), \\
	\mathcal{H}_{a}(0)&=a>0, \quad \mathcal{H}_{a}^{\prime}(0)=0.
	\end{align*}
	Set
	$
	c_{a}=\inf \{ s>0: \mathcal{H}_{a}(s)=0   \}. 
	$
	Then \\
	(1). $0<c_{a} \leq \frac{\pi}{2}. $ \\
	(2). For any $\epsilon \in (0,0.1)$, $\displaystyle c_{a} \geq \frac{\pi}{2(1+\epsilon)}$ if $a \in ( 0, \sqrt{2 \epsilon^{2}+4\epsilon}  )$. \\
	(3). $c_{a} \rightarrow 0$ as $a \rightarrow +\infty$. 
	
\end{prop} 
\begin{proof}
	(1) and (2) follow from an ODE comparison argument. 
	
	Set 
	$\displaystyle 
	\phi(s)=- \sin s \mathcal{H}_{a}(s)- \cos s \mathcal{H}_{a}^{\prime}(s).
	$
	Then for any $0<s<\min\{ c_{a}, \frac{\pi}{2} \}$, 
	\begin{align*}
	\phi^{\prime}(s)&=- \sin s \mathcal{H}_{a}^{\prime}(s)- \cos s \mathcal{H}_{a}(s)+ \sin s \mathcal{H}_{a}^{\prime}(s)- \cos s \mathcal{H}_{a}^{\prime \prime}(s) \\
	&=- \cos s (\mathcal{H}_{a}(s)+\mathcal{H}_{a}^{\prime \prime}(s)  )=\frac{1}{2} \cos s \mathcal{H}_{a}^{3}(s)>0,
	\end{align*}
	which implies
$\displaystyle 
	\phi(s) > \phi(0)=0, $
	then
	$\displaystyle 
	\frac{-\sin s}{\cos s} > \frac{ \mathcal{H}_{a}^{\prime}(s)}{\mathcal{H}_{a}(s)}.$
	By integration,
$\displaystyle 
	\ln \frac{ \cos s}{ \cos 0}> \ln \frac{ \mathcal{H}_{a}(s)}{\mathcal{H}_{a}(0)},$
	then, 
$\displaystyle 
	\frac{ \cos s}{\mathcal{H}_{a}(s) } > \frac{ \cos 0}{\mathcal{H}_{a}(0)} = \frac{1}{a}, 
$
	which implies $c_{a} \leq \frac{\pi}{2}$. 
	
	Similarly, set 
	$$
	\psi(s)=-(1+\epsilon) \sin((1+\epsilon)s) \mathcal{H}_{a}(s)- \cos ((1+\epsilon)s) \mathcal{H}_{a}^{\prime}(s).
	$$
	Then 
	\begin{align*}
	\psi^{\prime}(s)=-(1+\epsilon)^{2}  \cos((1+\epsilon)s) \mathcal{H}_{a}(s)- \cos((1+\epsilon)s) \mathcal{H}_{a}^{\prime \prime}(s).
	\end{align*}
	If $a \in (0,\sqrt{2 \epsilon^{2}+4\epsilon} )$, 
	\begin{align*}
	0=\mathcal{H}_{a}^{\prime \prime}(s)+\mathcal{H}_{a}(s)(1+\frac{1}{2} \mathcal{H}_{a}^{2}(s)) < \mathcal{H}_{a}^{\prime \prime}(s)+(1+\epsilon)^{2}\mathcal{H}_{a}(s), \forall s< c_{a}, 
	\end{align*}
	which implies $\psi^{\prime}(s) \leq 0$ for $\displaystyle s < \min\{ c_{a}, \frac{\pi}{2(1+\epsilon)}   \}$. 
	Since $\psi(0)=0$, 
	\begin{align*}
	-(1+\epsilon)  \sin((1+\epsilon)s) \mathcal{H}_{a}(s)- \cos ((1+\epsilon)s) \mathcal{H}_{a}^{\prime}(s)<0.
	\end{align*}
	Then
	$\displaystyle
	\frac{ \mathcal{H}_{a}^{\prime}(s)}{\mathcal{H}_{a}(s) } >-\frac{(1+\epsilon) \sin((1+\epsilon)s)}{ \cos((1+\epsilon)s) }. 
	$
	By integration, we obtain $\displaystyle \frac{\mathcal{H}_{a}(s)}{ \cos((1+\epsilon)s) } > \frac{ \mathcal{H}_{a}(0)}{ \cos 0 } = \frac{1}{a}$, which implies $\displaystyle c_{a} \geq \frac{\pi}{2(1+\epsilon)}$. 
	
	For (3), we consider $\displaystyle \hat{\mathcal{H}}_{a}(s)=\frac{1}{a} \mathcal{H}_{a}(\frac{s}{a})$. By direct calculations, 
	\begin{align*}
	\hat{\mathcal{H}}_{a}^{\prime \prime}(s)=-\hat{\mathcal{H}}_{a}(s)( \frac{1}{a^{2}}+\frac{1}{2} \hat{\mathcal{H}}_{a}^{2}(s)  ), \quad \hat{\mathcal{H}}_{a}(0)=1,\quad  \hat{\mathcal{H}}_{a}^{\prime}(0)=0. 
	\end{align*}
	Assume $\Psi$ satisfies 
	\begin{align*}
	\Psi^{\prime \prime}(s)=-\frac{\Psi^{3}(s)}{2}, \quad \Psi(0)=1, \quad \Psi^{\prime}(0)=0,  
	\end{align*}
	then $\hat{\mathcal{H}}_{a}$ converges locally uniformly to $\Psi$ as $a \rightarrow +\infty$. Set $\displaystyle \alpha=\min\{s>0| \Psi(s)=0\}$. By the symmetry and monotonicity of $\Psi$, we know that $\displaystyle \Psi(\frac{\alpha}{2})=-\Psi(\frac{3\alpha}{2})>0$. Then for $a$ sufficiently large, $\displaystyle \hat{\mathcal{H}}_{a}(\frac{\alpha}{2}  )>0, \hat{\mathcal{H}}_{a}(\frac{3\alpha}{2}  )<0$, which implies that there exists $\displaystyle\alpha_{a} \in (\frac{\alpha}{2a},\frac{3\alpha}{2a} )$ such that $\mathcal{H}_{a}(\alpha_{a})=0$. Then $c_{a} \leq \alpha_{a} \rightarrow 0$ as $a \rightarrow +\infty$. 
\end{proof}

\begin{prop}
	\label{prop  initial data property of H}
	Assuming  $\mathcal{H}:\mathbb{R} \rightarrow \mathbb{R}$ is a nonconstant function solving the equation:
	$\displaystyle
	\mathcal{H}^{\prime \prime}=-\mathcal{H}(1+\frac{1}{2} \mathcal{H}^{2}), 
	$
	$\mathcal{H}$ is periodic. 
	\end{prop}
\begin{proof}
	\textbf{Step 1:} Let $\mathcal{H}_{a}$ and $c_{a}$ be defined in Proposition \ref{prop Willmore cone ode comparison}. For any $x \leq c_{a}$, 
	\begin{align*}
	((\mathcal{H}_{a}^{\prime})^{2})^{\prime}=2\mathcal{H}_{a}^{\prime}\mathcal{H}_{a}^{\prime \prime}=- (\mathcal{H}_{a}^{2})^{\prime} (1+\frac{1}{2} \mathcal{H}_{a}^{2}).
	\end{align*}
	By integration, we obtain
	\begin{align*}
	(\mathcal{H}_{a}^{\prime})^{2}(x)=-(\mathcal{H}_{a}^{2}(x)+\frac{1}{4}\mathcal{H}_{a}^{4}(x))+a^{2}+\frac{a^{4}}{4},
	\end{align*}
	which implies $\displaystyle \mathcal{H}_{a}^{\prime}(c_{a})=-\sqrt{a^{2}+\frac{a^{4}}{4}}. $ We should note that $\displaystyle a \rightarrow a^{2}+\frac{a^{4}}{2}$ is a bijection between $\mathbb{R}^{+}$ and $\mathbb{R}^{+}$.
	 
	\textbf{Step 2:}
	By Liouville theorem, we may assume $\mathcal{H}(0)=0$. By the uniqueness of ODE, we also know that $\mathcal{H}$ is odd. We may further assume $\displaystyle \mathcal{H}^{\prime}(0)<0$ since $\mathcal{H} \equiv 0$ if $\mathcal{H}^{\prime}(0)=0$. Assume $\displaystyle \mathcal{H}^{\prime}(0)=-\sqrt{a^{2}+\frac{a^{4}}{4}}$ for some $a>0$ (unique). By 	\textbf{Step 1} and the uniqueness of ODE, there exists $c_{a}$ such that
	\begin{align*}
	\mathcal{H}^{\prime}(c_{a})=0, \quad \mathcal{H}(x)=\mathcal{H}( 2c_{a}-x ). 
	\end{align*}
	Therefore, $\mathcal{H}(x)=\mathcal{H}(x+4c_{a}), \forall x \in \mathbb{R}$. 
\end{proof}

	Now we assume $\gamma_{a}$ is the solution of $\displaystyle \nabla^{\mathbb{S}^{2}}_{\gamma_{a}^{\prime}} \gamma_{a}^{\prime}=\mathcal{H}_{a} \mathcal{J}(\gamma_{a}^{\prime})$ where $\mathcal{H}_{a}$ satisfies
	\begin{align*}
	\mathcal{H}_{a}^{\prime \prime}=-\mathcal{H}_{a}(1+\frac{1}{2} \mathcal{H}_{a}^{2}), \quad \mathcal{H}_{a}(0)=a>0, \quad \mathcal{H}_{a}^{\prime}(0)=0.
	\end{align*}
	By Proposition \ref{prop Willmore cone ode comparison} and Proposition \ref{prop initial data property of H}, there exists $c_{a} \leq \frac{\pi}{2}$ such that $\displaystyle \mathcal{H}_{a}(c_{a})=0, \mathcal{H}_{a}^{\prime}(c_{a})=-\sqrt{a^{2}+\frac{a^{4}}{4}}$. 
	
	Now we choose a coordinate system of $\mathbb{R}^{3}$ such that\\
	(1). $\gamma_{a}^{\prime}(0) $ is normal to $xz-plane$. \\
	(2). $y$-component of $\gamma_{a}(0)$ is 0, $z$-component of $\gamma_{a}(c_{a})$ is 0. 
	
	In this coordinate system, we write $\gamma_{a}(s)$ as $\{ ( \mathcal{X}(s), \mathcal{Y}(s), \mathcal{Z}(s) )   \}$. We set 
	\begin{align*}
	\tilde{\gamma}_{a}(s)=( \mathcal{X}(-s), -\mathcal{Y}(-s), \mathcal{Z}(-s)   ), 
	\end{align*}
	then by (1) and direct calculations, 
	\begin{align*}
	&\tilde{\gamma}_{a}(0)=\gamma_{a}(0), \quad \tilde{\gamma}_{a}^{\prime}(0)=\gamma_{a}^{\prime}(0), \\
	&	\tilde{\gamma}_{a}^{\prime \prime}(s) \cdot (	\tilde{\gamma}_{a}(s) \times 	\tilde{\gamma}_{a}^{\prime}(s))=\mathcal{H}_{a}(-s)=\mathcal{H}_{a}(s),
	\end{align*}
	 which implies $\gamma_{a}=\tilde{\gamma}_{a}$ by the uniqueness. 
	 
	 Now we select a new coordinate system $(\tilde{x},\tilde{y},\tilde{z})$ such that $\tilde{z}=z$ and $\tilde{x}$-component of $\gamma_{a}(c_{a})$ is 1. 
	 In this new coordinate system, we rewrite $\gamma_{a}(s)$ as $\{(X(s),Y(s),Z(s))\}$. Set
	$
	 \hat{\gamma}_{a}(s)=(X(2c_{a}-s), -Y(2c_{a}-s),-Z(2c_{a}-s)). 
	 $
	 We know that $\hat{\gamma}_{a}^{\prime}(0)$ is normal to $\tilde{x}$-axis, then
	 \begin{align*}
	 	& \hat{\gamma}_{a}(c_{a})=\gamma_{a}(c_{a}), \quad \hat{\gamma}_{a}^{\prime}(c_{a})=\gamma_{a}^{\prime}(c_{a}), \\
	 	& \hat{\gamma}_{a}^{\prime \prime}(s) \cdot (	\hat{\gamma}_{a}(s) \times 	\hat{\gamma}_{a}^{\prime}(s))=-\mathcal{H}_{a}(2c_{a}-s)=\mathcal{H}_{a}(s), 
	 \end{align*}
	 which implies $\gamma_{a}=\hat{\gamma}_{a}$ by the uniqueness again. 
	 
	Two symmetry properties obtained above are shown in Figure \ref{Figure 2}. 
	\begin{figure}[H]
		\centering
		\includegraphics[scale=0.5]{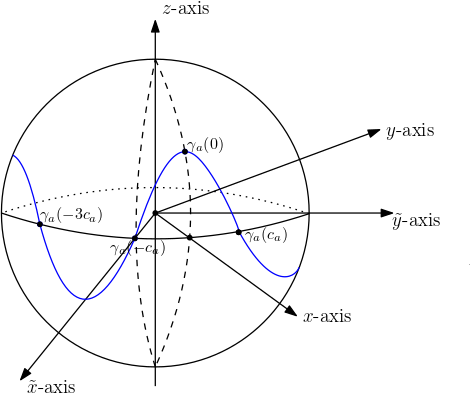}
		\caption{}
		\label{Figure 2}
	\end{figure}
	
We 	set $\displaystyle T_{a}= d_{\mathbb{S}^{2}}( \gamma_{a}(-c_{a}), \gamma_{a}(c_{a}) )$, $T_{a}$ is continuous with respect to $a$. By Proposition \ref{prop Willmore cone ode comparison} (3) and triangle inequality, 
	$\displaystyle
	\lim_{a \rightarrow +\infty} T_{a}=0$. By  Proposition \ref{prop Willmore cone ode comparison} (2), $\displaystyle c_{a} \rightarrow \frac{\pi}{2}$ and $\gamma_{a}$ tends to part of a great circle as $a \rightarrow 0$, which implies $\displaystyle
	\lim_{a \rightarrow 0+} T_{a}=\pi$. It is possible that for some $a$, $T_{a}=0$. However, if the image of $\gamma_{a}$ is a closed, embedded curve in $\mathbb{S}^{2}$, then $2m T_{a}=2 \pi $ for some integer $m \geq 1$. 
	The following theorem follows from the above argument immediately. 
	
	\begin{theorem}
		\label{theorem countable embedded Willmore cones}
		Assume $\mathscr{C}_{\gamma}$ is a Willmore cone. If the image of $\gamma$ is a closed, embedded curve in $\mathbb{S}^{2}$ and not a great circle, then the length of the image is strictly larger than $2 \pi$. 
	\end{theorem}

\section{A Bernstein-type theorem}

Throughout this section, we denote
\begin{align*}
D_{r}(x)=\{ y \in \mathbb{R}^{2}: |y-x| \leq r  \}, \quad D_{r}=D_{r}(0). 
\end{align*}

We firstly recall two technical lemmas. 
\begin{lemma}
	\label{lemma Helein intrisic theorem}
	Let $g=e^{2u} g_{\mathbb{R}^{3}}$ be a metric defined on $D_{1}$, and $K_{g}$ be the Gaussian curvature. 
	We  assume for some constant $\Lambda>0$ and any geodesic ball $B_{r}^{g}(x) $, 
	\begin{align}
	\label{BTT1}
	\mu_{g}(B_{r}^{g}(x))   \leq \Lambda r^{2}.
	\end{align}
	Then, there exists $\epsilon_{0}(\Lambda) >0$ such that if 
	\begin{align}
	&\int_{D_{1}} |K_{g}|d \mu_{g}<\epsilon_{0},  \label{BTT2} 
	\end{align}
	then for any $q \in [1,2)$, 
	there exists $C(\epsilon_{0},q)$ such that 
	$$
	\int_{D_{r}(x)} |\nabla u|^{q} d \mu_{\mathbb{R}^{2}} \leq C r^{2-q}, \quad \forall D_{r}(x) \subset D_{\frac{1}{2}}.
	$$
\end{lemma}
 \begin{proof}
 	See \cite[Lemma $2.5$, Lemma $3.1$]{Li2019MetricsOA} and \cite[Proposition $2.2$]{chenli2022}.
 	\end{proof}
 
 \begin{lemma}
 	\label{lemma harmonic function on punctured disk}
 	If $u$ is a harmonic function on $D_{1} \setminus \{0\}$, 
 	$$
 	u(z)=Re( h(z)  )+c \log |z|, z \in D_{1} \setminus \{0\},
 	$$
 	where $h(z)$ is holomorphic on $D_{1} \setminus \{0\}$. 
 \end{lemma}
 \begin{proof}
 See \cite[Theorem 15.1.3]{conway2012functions} .
 \end{proof}

 Now we come to prove Theorem \ref{theorem berstein-type main theorem}. 
 \begin{proof}
 	We may assume $f^{*} g_{\mathbb{R}^{3}}=e^{2u} g_{\mathbb{R}^{2}}.$ Let $g=e^{2u}g_{\mathbb{R}^{2}}$ and $K_{g}$ be the corresponding Gaussian curvature. 
 	
 	\textbf{Step 1:}
 	For $r_{k}$ with $\displaystyle \lim_{k \rightarrow \infty} r_{k} =+\infty$, set $u_{k}(x)=u(r_{k} x)-c_{k}$ where $c_{k}$ is to be determined later, $g_{k}=e^{2u_{k}} g_{\mathbb{R}^{2}}$ and $\mu_{k}$ is the measure induced by $g_{k}$.  Since $\displaystyle \int_{\mathbb{R}^{2}} |K_{g}| d \mu_{g} \leq C$, there exists $\Lambda>0$ such that for any geodesic ball $B_{r}^{g}(x)$, 
 	$\displaystyle
 	\mu_{g}(B_{r}^{g}(x)) \leq \Lambda r^{2}$. 
 	By direct calculations, 
 	\begin{align*}
 	-\Delta u_{k}(x)&=-r_{k}^{2} \Delta u(r_{k} x)=-r_{k}^{2 } e^{2u( r_{k}x )}K_{g}(r_{k}x)=-r_{k}^{2} e^{2u_{k}(x)+2c_{k}}K_{g}(r_{k}x) , 
 	\end{align*}
 	then for any fixed $R>100$, 
 	\begin{align*}
 	&\int_{D_{2R} \setminus D_{\frac{1}{2R}}} |K_{g_{k}}(x)| d\mu_{k}(x)=
 	\int_{D_{2R} \setminus D_{\frac{1}{2R}}} |K_{g_{k}}(x)|e^{2u_{k}(x)} d\mu_{\mathbb{R}^{2}}(x)\\=&\int_{D_{2R} \setminus D_{\frac{1}{2R}}} r_{k}^{2} e^{2u_{k}(x)+2c_{k}}|K_{g}(r_{k}x) | d\mu_{\mathbb{R}^{2}}(x) 
 	=\int_{D_{2Rr_{k}} \setminus D_{\frac{r_{k}}{2R}}} |K_{g}(x)| e^{2u(x)} d\mu_{\mathbb{R}^{2}}(x) \rightarrow 0, \text{ as } k \rightarrow \infty. 
 	\end{align*}
 	
 Therefore, for sufficiently large $k$, $\int_{D_{4} \setminus D_{\frac{1}{2}}} |K_{g_{k}}| d\mu_{k} \leq \epsilon_{0}$, where $\epsilon_{0}$ is chosen as in Lemma \ref{lemma Helein intrisic theorem}. 
 	For any $D_{r}(x) \subset \mathbb{R} \setminus \{0\}$ with $\displaystyle r<\frac{|x|}{3}$, choose $\rho$ such that $D_{r}(x) \subset D_{2\rho} \setminus D_{\rho}$. 
 	Set $v_{k}(x)=u_{k}(\rho x)$. Then $D_{ \frac{r}{\rho}  }(\frac{x}{\rho}) \subset D_{2} \setminus D_{1}$ and 
 	$\displaystyle
 	\int_{ D_{\frac{r}{\rho}}(\frac{x}{\rho})} |K_{g_{v_{k}}}| d \mu_{v_{k}}=\int_{ D_{r}(x)} |K_{g_{k}}| d \mu_{k},
 	$
 	where $\mu_{v_{k}}$ and $K_{g_{v_{k}}}$ are induced by $v_{k}$. By Lemma \ref{lemma Helein intrisic theorem}, 
 	\begin{align*}
 	&r^{q-2} \int_{ D_{r}(x)} |\nabla u_{k}(y)|^{q} d \mu_{\mathbb{R}^{2}}(y)
 	=r^{q-2} \int_{ D_{ \frac{r}{\rho}  }(\frac{x}{\rho})} |\nabla u_{k}(\rho z)|^{q} \rho^{2} d \mu_{\mathbb{R}^{2}}(z) \\
 	=& r^{q-2} \int_{ D_{ \frac{r}{\rho}  }(\frac{x}{\rho})}  |\nabla v_{k}(z)|^{q}  \rho^{2-q} d \mu_{\mathbb{R}^{2}}(z) \leq C. 
 	\end{align*}
 	
 	In conclusion, combining with a covering argument, we obtain for any $q \in [1,2)$, for any $R>0$, 
 	\begin{align}
 	&\int_{ D_{R} \setminus D_{\frac{1}{R}}} |\nabla u_{k}|^{q} d \mu_{\mathbb{R}^{2}} \leq C(\epsilon_{0},R,q), \label{BD1}\\
 	&\int_{ D_{r}(x) } |\nabla u_{k}|^{q} d \mu_{\mathbb{R}^{2}} \leq C(\epsilon_{0},q) r^{2-q}, \quad \forall D_{r}(x) \subset \mathbb{R}^{2} \setminus \{0\} \text{ with } r<\frac{|x|}{3}  . \label{BD2}
 	\end{align}

 	Now we take $\displaystyle c_{k}=\frac{1}{ |D_{2} \setminus D_{\frac{1}{2}} |   } \int_{ D_{2} \setminus D_{\frac{1}{2} }} u(r_{k}x) d \mu_{\mathbb{R}^{2}}(x)$ which is independent of $R$, then by Poinca$\acute{r}$e inequality(see  \cite[Theorem 5.4.3]{attouch2014variational}) and (\ref{BD1}), 
 	\begin{align*}
 	\int_{D_{R} \setminus D_{\frac{1}{R}}} |u_{k}(x)| d \mu_{\mathbb{R}^{2}}(x)
 	\leq  C [  \int_{D_{R} \setminus D_{\frac{1}{R}}} |\nabla u_{k}(x)|^{q}   d \mu_{\mathbb{R}^{2}}(x)      ]^{\frac{1}{q}}  \leq C(\epsilon_{0},R,q),
  	\end{align*}
  	which implies $u_{k}$ converges to some $u_{\infty}$ weakly in $W^{1,q}( D_{R} \setminus D_{\frac{1}{R}}  )$ and $-\Delta u_{\infty}=0$ weakly in $D_{R} \setminus D_{\frac{1}{R}}$ since $K_{g_{k}}e^{2u_{k}}$ converges to 0 in $L^{1}( D_{R} \setminus D_{\frac{1}{R}}  )$.
  	By the arbitrariness of $R$, we conclude that $u_{k}$ converges to $u_{\infty}$ weakly in $W_{loc}^{1,q}(\mathbb{R}^{2} \setminus \{0\})$, $u_{\infty}$ is harmonic on $\mathbb{R}^{2} \setminus \{0\}$. By the semi-continuity of weak convergence and (\ref{BD2}), 
  	$$
  	\int_{ D_{r}(x)} | \nabla u_{\infty} |^{q}  \leq \liminf_{k \rightarrow \infty}  \int_{D_{r}(x)} |\nabla u_{k}|^{q} \leq C(\epsilon_{0},q)r^{2-q}, \quad \forall D_{r}(x) \subset \mathbb{R}^{2} \setminus \{0\} \text{ with } r<\frac{|x|}{3}.
  	$$ 
  	
  	Fix $R>0$, by Lemma \ref{lemma harmonic function on punctured disk}, there exists $h$ holomorphic on $D_{R} \setminus \{0\}$ and $c$ such that
  	$$
  	u_{\infty}(x)=Re(h(x))+c \log |x|, \quad \forall x \in D_{R} \setminus \{0\}. 
  	$$
  	Since $\log z$ is not holomorphic on $D_{R} \setminus \{0\}$, $h$ and $c$ do not depend on the choice of $R$. 
  	For any $D_{r}(x) \subset D_{R} \setminus \{0\}$ with $\displaystyle r<\frac{|x|}{3}$, 
  	\begin{align*}
  	\int_{D_{r}(x)} |\nabla \log|z| |^{q} =\int_{D_{r}(x)} |\frac{z}{|z|^{2}}|^{q} \leq \frac{ \pi r^{2}}{|z|^{q}} \leq C\pi r^{2-q}.
  	  	\end{align*}
  	  	Then 
  	  	\begin{align*}
  	  	|\nabla h(x)| &\leq \frac{1}{\pi r^{2}}\int_{D_{r}(x)} |\nabla h(z)| \leq C(\int_{D_{r}(x)} |\nabla h(z)|^{q})^{\frac{1}{q}} r^{2-\frac{ 2}{q} } r^{-2} \leq \frac{C}{r}, 
  	  	\end{align*}
 which implies $|x| |h^{\prime}(x)| \leq C$. Consider the Laurent expansion of $h^{\prime}$:
 \begin{align*}
 h^{\prime}(z)=\sum_{k=-\infty}^{+\infty} a_{k} z^{k}, \quad a_{-1}=0, \quad a_{k}=\frac{1}{2\pi \textbf{i} } \int_{ |z|=\rho} \frac{ h^{\prime}(z)}{|z|^{k+1}}, \forall \rho>0. 
 \end{align*}
  	  Since $\displaystyle |a_{k}| \leq \frac{C}{2 \pi} \int_{ |z|=\rho} |z|^{-k-2} dx \leq C \rho^{-k-1}$, we have $a_{k}=0$ for $k\neq -1$. Then $h$ is constant.  
  	  
  	  In conclusion, $u_{\infty}(x)=c_{1} \log|x|+c_{2}, \forall x \in \mathbb{R}^{2} \setminus \{0\}$ for some constants $c_{1}$ and $c_{2}$.  
  	  
  	  	\textbf{Step 2:}
  	  	Assume $f(0)=0$ and set
  	  	$ \displaystyle 
f_{k}(x)=\frac{f( r_{k}x ) 	  	}{r_{k} e^{c_{k}} }.
  	  	$
  	  	Then
  	  	\begin{align*}
  	  	|\nabla f_{k}(x)|^{2} =|\nabla f(r_{k}x)|^{2} e^{-2c_{k}}=2e^{2u(r_{k}x)-2c_{k}  }=2e^{2u_{k}(x)}. 
  	  	\end{align*}
  	  
  	  	Set $\displaystyle h_{k}(x)=\inf \{ r : \int_{D_{r}(x)} |A_{k}|^{2} d \mu_{k} \geq \gamma_{0}     \}$, where $\gamma_{0}$ is not larger than $\gamma$ in \cite[Theorem 5.1.1]{helein2002harmonic} and $\epsilon$ in \cite[Theorem  \uppercase\expandafter{\romannumeral1}.5]{RT08}. 
  	  	 
  	  	If for some $R$ sufficiently large, there exists $y_{k} \in D_{R} \setminus D_{\frac{1}{R}}$ such that $h_{k}(y_{k} ) \rightarrow 0$. 
  	  	 
  	  	 Set $\displaystyle \rho_{k}=\inf_{x \in D_{2R} \setminus D_{\frac{1}{2R}}   }  \frac{ h_{k}(x)}{ \min\{ |x|-\frac{1}{2R}, 2R-|x|     \}       }$. Since 
  	  	 \begin{align*}
  	  	 \lim_{|x| \rightarrow \frac{1}{2R} or |x| \rightarrow 2R}  \frac{ h_{k}(x)}{ \min\{ |x|-\frac{1}{2R}, 2R-|x|     \}       }=+\infty, 
  	  	 \end{align*}
  	  	  there exists $\displaystyle x_{k} \in D_{2R} \setminus D_{\frac{1}{2R}}  \setminus (\partial D_{2R} \cup \partial D_{\frac{1}{2R}})$ such that 
  	  	  $\displaystyle \rho_{k}=\frac{ h_{k}(x_{k})}{ \min\{ |x_{k}|-\frac{1}{2R}, 2R-|x_{k}|     \}       }$. By our choice, 
  	  	  \begin{align*}
  	  	  \rho_{k} \leq  \frac{ h_{k}(y_{k})}{ \min\{ |y_{k}|-\frac{1}{2R}, 2R-|y_{k}|     \}       } \leq R h_{k}(y_{k}) \rightarrow 0, 
  	  	  \end{align*}
  	  	which implies $h_{k}(x_{k} ) \rightarrow 0$. For convenience, denote $\beta_{k}=h_{k}(y_{k})$. Set 
  	  	\begin{align*}
  	  	\displaystyle F_{k}(x)=e^{-\alpha_{k}} ( f_{k}(x_{k}+ \beta_{k} x ) -f_{k}(x_{k})      )   , \quad U_{k}(x)=u_{k}(x_{k}+\beta_{k}x)-\alpha_{k}  +\log \beta_{k}, 
  	  	  \end{align*}
  	  	where $\alpha_{k}$ is chosen later. Clearly, $|\nabla F_{k}|^{2}=2e^{2U_{k}}$. Let $\tilde{g}_{k}$ and $\tilde{\mu}_{k}$ are induced by $F_{k}$. 
  	  	For any $\rho>0$, by the choice of $\beta_{k}$, 
  	  	\begin{align*}
  	  	\int_{D_{\rho}} |K_{\tilde{g}_{k}}| d \tilde{\mu}_{k}=\int_{D_{ \beta_{k} \rho}(x_{k})} |K_{g_{k}}| d \mu_{k} \leq \int_{D_{2R} \setminus D_{\frac{1}{2R}}} |K_{g_{k}}| d \mu_{k} \rightarrow 0, \text{ as } k \rightarrow +\infty. 
  	  	\end{align*}
  	  Repeating the argument in \textbf{Step 1}, we could obtain for any $q \in [1,2)$, 
  	  \begin{align*}
  	  & \int_{D_{\rho}} |\nabla U_{k}|^{q} d \mu_{\mathbb{R}^{2}}\leq C(\rho), \\
  	  & \int_{D_{r}(x)} |\nabla U_{k}|^{q} d \mu_{\mathbb{R}^{2}} \leq C(q) r^{2-q}, \forall D_{r}(x) \subset \mathbb{R}^{2} .
  	  \end{align*}
  We take $\displaystyle \alpha_{k}=\frac{1}{|D_{1}|} \int_{D_{1}} ( u_{k}(x_{k}+\beta_{k}x)+\log \beta_{k}  )$. By Poinca$\acute{r}$e inequality, we conclude that $U_{k}$ converges to a harmonic function $U_{\infty}$ weakly in $W^{1,q}_{loc}(\mathbb{R}^{2})$. By the semi-continuity of weak convergence, 
  \begin{align*}
  \int_{D_{r}(x)} |\nabla U_{\infty}|^{q} \leq C(q) r^{2-q}, \quad \forall D_{r}(x) \subset \mathbb{R}^{2} , 
  \end{align*}
  which implies $|\nabla U_{\infty}| \equiv 0$. Therefore, $U_{\infty} \equiv 0$ since $\displaystyle \int_{D_{1}} U_{\infty}=0$. 
  
  By direct calculations, for any $D_{\rho}(x) $, $\displaystyle \int_{D_{\rho}(x)} |\tilde{A}_{k}|^{2} d \tilde{\mu}_{k}=\int_{ D_{\beta_{k} \rho}(x_{k}+\beta_{k} x ) } |A_{k}|^{2} d \mu_{k}$. If we set $\displaystyle \tilde{h}_{k}(x)=\inf\{r: \int_{D_{r}(x)} |\tilde{A}_{k}|^{2 } d \tilde{\mu}_{k} \geq \gamma_{0}  \}$, we have $\beta_{k} \tilde{h}_{k}(x)=h_{k}(x_{k}+\beta_{k}x)$. By the choice of $\rho_{k}$ and $\beta_{k}$, $D_{\beta_{k} \rho}(x_{k}+\beta_{k} x  )  \subset D_{2R} \setminus D_{\frac{1}{2R}}$, 
  \begin{align*}
  \tilde{h}_{k}(x)&=\frac{1}{\beta_{k}}h_{k}(x_{k}+\beta_{k}x)=\frac{1}{\beta_{k}}h_{k}(x_{k}+\beta_{k}x) \cdot \frac{ \min\{ |x_{k}+\beta_{k}x|-\frac{1}{2R}, 2R-|x_{k}+\beta_{k}x | \}  }{\min\{ |x_{k}+\beta_{k}x|-\frac{1}{2R}, 2R-|x_{k}+\beta_{k}x | \}}\\
  & \geq \frac{1}{\beta_{k}} \beta_{k} \cdot \frac{ \min\{ |x_{k}|-\frac{1}{2R}, 2R-|x_{k}| \}  }{\min\{ |x_{k}+\beta_{k}z|-\frac{1}{2R}, 2R-|x_{k}+\beta_{k}z | \}} \geq \frac{1}{2}.
  \end{align*}
  By  Hélein's convergence theorem(see  \cite[Theorem 5.1.1]{helein2002harmonic}) and \cite[Theorem  \uppercase\expandafter{\romannumeral1}.5]{RT08}, we conclude that $F_{k}$ converges to some $F_{\infty}$ (not a point since $U_{\infty}$ is constant) smoothly on $\mathbb{R}^{2}$ and $K_{\tilde{g}_{\infty} } \equiv 0$. By Theorem \ref{theorem intro}, $F_{\infty}$ is a plane, a contradiction to $\tilde{h}_{k}(0) \equiv 1$.

  In conclusion, for any $R$ sufficiently large, for any $\displaystyle x \in D_{R} \setminus D_{\frac{1}{R}}$, $h_{k}(x) > \mathscr{R}(R)>0$, then by  Hélein's convergence theorem, we know that $f_{k}$ converges weakly in $W^{2,2}(D_{ \frac{R}{2}} \setminus D_{\frac{2}{R}})$ since $u_{k}$ does not converge to $-\infty$. By the arbitrariness of $R$ and \cite[Theorem  \uppercase\expandafter{\romannumeral1}.5]{RT08}, $f_{k}$ converges smoothly to some $f_{\infty}$ on $\mathbb{R} \setminus  \{0\}$ with $|\nabla f_{\infty}|^{2}=2e^{2u_{\infty}}$.

      	\textbf{Step 3:}
      	By divergence theorem, 
      	for any $R_{1}>R_{2}>0$, we have
      	\begin{align*}
      	&\int_{\partial D_{R_{1}}} \frac{ \partial u_{k}}{\partial r}-\int_{\partial D_{R_{1}}} \frac{ \partial u_{k}}{\partial r}=\int_{\partial D_{r_{k}R_{1}}} \frac{ \partial u}{\partial r}-\int_{\partial D_{r_{k}R_{2}}} \frac{ \partial u}{\partial r} \\
      	=&\int_{ D_{r_{k}R_{1}} \setminus D_{r_{k}R_{2} }} \Delta u d \mu_{\mathbb{R}^{2}}=-\int_{ D_{r_{k}R_{1}} \setminus D_{r_{k}R_{2} }} K_{g} d \mu_{g},
      	\end{align*}
      	which implies $\displaystyle \lim_{R \rightarrow \infty} \int_{\partial D_{R}} \frac{ \partial u_{k}}{\partial r} $ exists for any $k$.
    We also have
      	\begin{align*}
      	&	\int_{\mathbb{R}^{2}} K_{g} d \mu_{g} =-\int_{\mathbb{R}^{2}} \Delta u d \mu_{\mathbb{R}^{2}}=-\lim_{R \rightarrow \infty} \int_{D_{r_{k} R}} \Delta u d \mu_{\mathbb{R}^{2}} \\
      =&	-\lim_{R\rightarrow \infty}  \int_{ \partial D_{R}} \frac{ \partial u_{k}}{ \partial r}  .
      	\end{align*} 
      	Since 
      	\begin{align*}
      \int_{R_{1}}^{R_{2}} \int_{0}^{2 \pi} | \frac{ \partial u_{\infty}}{\partial r} - \frac{ \partial u_{k}}{\partial r}  | d \theta dr \longrightarrow 0, \text{ as } k \rightarrow \infty, 
       	\end{align*}
      	then by the arbitrariness of $R_{1}$ and $R_{2}$, we have
      	\begin{align*}
      \int_{ \partial D_{R}}	\frac{ \partial u_{k}}{\partial r}  \longrightarrow  \int_{ \partial D_{R}} \frac{ \partial u_{\infty}}{\partial r}  \text{ a.e. } R  \in (0,\infty).
      	\end{align*}
      	Therefore, 
      	\begin{align*}
      	0 \geq \lim_{R \rightarrow \infty}  \int_{ \partial D_{R}} \frac{ \partial u_{\infty}}{\partial r}=2 \pi c_{1} \geq -2 \pi,
      	\end{align*}
      	which impies $-1\leq  c_{1} \leq 0.$
    
      		\textbf{Step 4:} If $c_{1} \in (-1,0)$, set $A=c_{1}+1$, $\rho=r^{A}$ and  $g_{A}=A^{2} r^{2A-2} g_{\mathbb{R}^{2}}$. We have
      			\begin{align*}
      	g_{A}&=A^{2} r^{2A-2} ( dr^{2}+r^{2} d \theta^{2}  ) \\
      		&=A^{2}  \rho^{ 2-\frac{2}{A}  } ( \frac{\rho^{ \frac{2}{A}-2  }}{A^{2}} d \rho^{2}+\rho^{ \frac{2}{A} }   d \theta^{2}  )=d \rho^{2} +A^{2} \rho^{2} d \theta^{2}. 
      		\end{align*}
      		Then  we can view $(\mathbb{R}^{2} \setminus \{0\},g_{A})$ as $\displaystyle \{ r \theta: r>0, 0 \leq \theta \leq 2A \pi\}$ by gluing the boundary. Set $\displaystyle \Sigma=f_{\infty}( \mathbb{R}^{2} \setminus \{0\}  )$ and up to multiplying a constant, we may assume $f_{\infty}^{*}(g_{\mathbb{R}^{3} })=g_{A}$ and $\lim_{|x| \rightarrow 0} f_{\infty}(0)=0$. 
      		
      		The following argument is similar as in the proof of \cite[Theorem $3$]{Hartman1959OnSI}. Since $c_{1} \neq 0$, then by Lemma \ref{lemma rank 1 implies line} and Corollary \ref{cor rank 1 implies line}, there exists $p \in \Sigma$ with $r^{*}(p)=1$ and $\Sigma$ contains part of a line $l(p)$ which passing $p$. If both ends of $l(p)$ can extend to infinity, $f_{\infty}^{-1}(l(p))$ is a geodesic line not passing the origin. 
      		
      		(Case 1). If $c_{1} \in (-1,-\frac{1}{2})$, $f_{\infty}^{-1}(l(p))$ will intersect itself, which leads to a contradiction. This means for any point of $\Sigma$ with $r^{*}(p)=1$, $l(p)$ is a ray starting from the origin and $f_{\infty}^{-1}(l(p))$ is also a ray starting from the origin. For a point $q$ with $r^{*}(q)=0$, let $\hat{l}(q)$ be the ray passing $f_{\infty}^{-1}(q)$. By Lemma \ref{lemma rank 1 implies line}, $r^{*}|_{ f_{\infty}(\hat{l}(q))   }=0$, then $f_{\infty}(\hat{l}(q))   $ is a ray starting from the origin by the isometry. In conclusion, $\Sigma$ is a Willmore cone. However, $l(\partial B_{1}^{g_{A}} (0) )=2 \pi A<2 \pi$, a contradiction to Theorem \ref{theorem countable embedded Willmore cones}.
      		
      		(Case 2). If $c_{1} \in [-\frac{1}{2},0)$, we can choose $\theta_{0}$ such that $\{r \theta_{0}:r>0\}$ is parallel to $f_{\infty}^{-1}(l(p))$. Set $\Omega=\{ r \theta:r>0, \theta_{0} <\theta<\theta_{0}+\pi \}$. By Lemma \ref{lemma rank 1 implies line} and isometry, $f_{\infty}(\Omega)$ is a cylinder, $f_{\infty}(\Omega^{c})$ is part of a Willmore cone, and the boundary of  $f_{\infty}(\Omega^{c})$ is a line, which implies $l(\partial B_{1}^{g_{A}} (0) ) \geq \pi+\pi$. However, $l(\partial B_{1}^{g_{A}} (0) )=2 \pi A<2 \pi$, which leads to a contradiction. Therefore, $\Sigma$ is a Willmore cone, a contradiction again by Theorem \ref{theorem countable embedded Willmore cones}.
      		
      		If $c_{1}=-1$. $(\mathbb{R}^{2} \setminus \{0\}, \frac{dr^{2}}{r^{2}}+d \theta^{2})$ is indeed $S^{1} \times \mathbb{R}$, then $\Sigma$ is a nontrivial embedded Willmore surface in $\mathbb{R}^{3}$. Note that the conclusion in  Step 1 of the proof of Theorem \ref{theorem intro} still holds  here by Liouville theorem and we can repeat the argument as in the proof of Theorem \ref{theorem intro}, which leads to a contradiction to the assumption  	($\mathscr{M}2$). 
      		
      		We conclude that $\displaystyle  c_{1}=0$. Then as the argument in \textbf{Step 3}, $\displaystyle \int_{\mathbb{R}^{2}} K_{g} d \mu_{g}=0$, which leads to a contradiction to the assumption ($\mathscr{M}1$), we complete the proof. 
 \end{proof}

 \section*{Appendix}
\cite[Lemma $2$]{Hartman1959OnSI}   plays an essential role in \cite{Hartman1959OnSI}, we give a proof of a 2-dimensional case here for the convenience of the reader. 

 Let $\Omega \subset \mathbb{R}^{2}$ be a domain. Let $f:\Omega  \rightarrow \mathbb{R}$ be a $C^{2}$ function. For convenience, set $p=\nabla f=(p_{1}(x_{1},x_{2}),p_{2}(x_{1},x_{2}) )$. Clearly, $p$ is a gradient map defined as in \cite{Hartman1959OnSI}. 
 
 Let $J(x)$ denote the Jacobian matrix $\displaystyle ( \frac{ \partial p_{i} }{ \partial x_{j}} )_{i,j}^{}$, $r(x)$ denote the rank of $J(x)$ 
 and $r^{*}(x)$ denote the largest interger $s$ with the property that every neighbourhood of $x$ contains a point $x^{*}$ with $r(x^{*})=s$. By a standard calculation, we know that the Gaussian curvature of the graph induced by $f$ vanishes if and only if $r \leq 1$. If $r^{*}(x)=0$,  there exists a neighbourhood $U$ of $x$ such that $\{(x,f(x)) : x \in U\}$ is contained in a plane.

 \begin{lemma}
 	\label{lemma rank 1 implies line}
 Assume  $\displaystyle \det J(x)$ is identically zero and  at a point $x^{0}=(x_{1}^{0},x_{2}^{0})$, $r(x^{0})=1$. Set $S=\{ x \in \Omega: r(x)=1  \}$. Then $p(x)$ is constant on  $l(x^{0}) \cap S$, where $l(x^{0}) \subset \mathbb{R}^{2}$ is a unique straight line passing $x^{0}$. Furthermore, $r(x)=1$ on $l(x^{0}) \cap S$. 
 \end{lemma}
 
 \begin{proof}
 	\textbf{Step 1:} We show that in a small neighbourhood $U$ of $x^{0}$, $\{x \in U: p(x)=p(x^{0})   \}$ is the intersection of a straight line passing $x^{0}$ and $U$. 
 	
 	Since $\det J(x^{0})=0, r(x^{0})=1$, we may assume $\displaystyle \frac{\partial p_{1}}{\partial x_{1}}|_{x^{0}} \neq 0$ up to a linear transformation of the coordinates. Consider the map $q(x_{1},x_{2})=(q_{1}(x_{1},x_{2}), q_{2}(x_{1},x_{2})  )=( p_{1}(x_{1},x_{2}), x_{2}    )$.The corresponding Jacobian matrix is
 	$
 	\begin{pmatrix}
 	\frac{\partial p_{1}}{\partial x_{1}} & 0 \\ \frac{\partial p_{1}}{\partial x_{2}} & 1
 	\end{pmatrix}
 	$, then we can introduce $(q_{1},q_{2})$ as new coordinates in $q(U)$ up to contracting $U$. Then 
 	$$
 	x_{2}(q_{1},q_{2})=q_{2}; \quad p(q_{1},q_{2})=(q_{1}, p_{2}(q_{1},q_{2})  ). 
 	$$
 	
 	By chain rules, we obtain
 	$$
 	\begin{pmatrix}
 	1  & 0 \\ \frac{\partial p_{2}}{\partial q_{1}} & \frac{\partial p_{2}}{\partial q_{2}} 
 	\end{pmatrix}
 	=
 	\begin{pmatrix}
 	\frac{\partial p_{1}}{\partial q_{1}}   & \frac{\partial p_{1}}{\partial q_{2}}  \\ \frac{\partial p_{2}}{\partial q_{1}} & \frac{\partial p_{2}}{\partial q_{2}} 
 	\end{pmatrix}
 	=
 	\begin{pmatrix}
 	\frac{\partial p_{1}}{\partial x_{1}}   & \frac{\partial p_{1}}{\partial x_{2}}  \\ \frac{\partial p_{2}}{\partial x_{1}} & \frac{\partial p_{2}}{\partial x_{2}} 
 	\end{pmatrix}
 	\begin{pmatrix}
 	\frac{\partial x_{1}}{\partial q_{1}}   & \frac{\partial x_{1}}{\partial q_{2}}  \\ 0& 1
 	\end{pmatrix},
 	$$
 	By the assumption, we have $\displaystyle \frac{\partial p_{2}}{ \partial q_{2}}=0$, which implies $p_{2}$ depends only on $q_{1}$ and is independent of $q_{2}$. Also by $\displaystyle 
 	\frac{ \partial p_{1}}{ \partial x_{2} }=\frac{ \partial p_{2}}{ \partial x_{1} }, 
 	$ we obtain
 	$$
 	\frac{\partial x_{1}}{\partial q_{2}}+\frac{\partial p_{2}}{\partial q_{1}}=\frac{\partial p_{1}}{\partial q_{2}}\frac{\partial x_{1}}{\partial q_{1}}=0,
 	$$
 	then we obtain for some function $\beta(q_{1})$, 
 	\begin{align}
 	\label{HN1}
 	\beta(q_{1})=x_{1}+\frac{\partial p_{2}}{\partial q_{1}}q_{2}=x_{1}+\alpha(q_{1})x_{2}, 
 	\end{align}
 	where we set $\alpha=\frac{\partial p_{2}}{\partial q_{1}}$. 
 	
 	Now for any $x=(x_{1},x_{2})$ such that $p(x)=p(x^{0})$, $q_{1}=q_{1}^{0}$ for some $q_{1}^{0}$, then locally, 
 	$\displaystyle p(x)=p(x^{0}) $
 	is equivalent to $\displaystyle \beta(q_{1}^{0})=x_{1}+\alpha(q_{1}^{0}) x_{2}$, 
 	which implies in $U$, $\{x \in U: p(x)=p(x^{0})   \}$ is the intersection of a straight line passing $x^{0}$ and $U$. 
 	
 	\textbf{Step 2:} Let $l(x^{0})=\{ (x_{1},x_{2}) \in \Omega: \beta(q_{1}^{0})=x_{1}+\alpha(q_{1}^{0}) x_{2}    \}$, we will show that $p$ is constant on $l(x^{0}) \cap S$. By (\ref{HN1}), 
 	\begin{align}
 	1=\frac{ \partial p_{1}}{\partial q_{1}}&=\frac{ \partial p_{1}}{\partial x_{1}}\frac{ \partial x_{1}}{\partial q_{1}}+\frac{ \partial p_{1}}{\partial x_{2}}\frac{ \partial x_{2}}{\partial q_{1}} \nonumber\\
 	&=\frac{ \partial p_{1}}{\partial x_{1}} \frac{ \partial ( -\alpha(q_{1}) x_{2}+\beta(q_{1})    )}{ \partial q_{1}  } \nonumber\\
 	&=\frac{ \partial p_{1}}{\partial x_{1}} (-x_{2} \frac{\partial \alpha}{\partial q_{1}  }+   \frac{\partial \beta}{\partial q_{1}  }         ) \label{HN2}. 
 	\end{align}
 	
 	We claim that $p(x)=p(x^{0})$ and $\displaystyle \frac{\partial p_{1}}{\partial x_{1}} \neq 0$ on $l(x^{0}) \cap S$. Assuming the contrary, there exists a curve $\gamma:[0,1] \rightarrow l(x^{0}) \cap S$ with $\gamma(0)=x^{0}, \gamma(1) \in \partial S$. 
 	Set
 	$$
 	t_{0}=\sup \{  s: p(\gamma(t))=p(x^{0}), \frac{ \partial p_{1}}{\partial x_{1}} (\gamma(t))\neq 0, t \leq s    \} <1.
 	$$
 	Clearly, $p(\gamma(t_{0}))=p(x^{0})$. Furthermore, (\ref{HN2}) holds in a small neighbourhood of each point $\gamma(s)$ for $s<t_{0}$. Along $\gamma|_{[0,s]}$, $\displaystyle  \frac{\partial \alpha}{\partial q_{1}  }, \frac{\partial \beta}{\partial q_{1}  }$ is independent of $s$, which implies
 	$\displaystyle \frac{\partial p_{1}}{\partial x_{1}}(\gamma(s))$ is bounded away from zero as $s \rightarrow t_{0}$. Applying the argument in \textbf{Step 1} at the point $p( \gamma(t_{0}))$, we will obtain a contradiction to the maximality of $t_{0}$. We complete the proof. 
 \end{proof}
 The following corollary follows from Lemma \ref{lemma rank 1 implies line} and an approximation argument immediately. 
 \begin{cor}
 	\label{cor rank 1 implies line}
 	If $r^{*}(x^{0})=1$ at a point $x^{0} \in \Omega$, then $p(x)$ is constant on the intersection of a straight line passing $x^{0}$ and $\Omega$. 
 \end{cor}

\nocite{*}
\bibliography{WS}

\begin{bibdiv}
\begin{biblist}

\bib{attouch2014variational}{book}{
      author={Attouch, Hedy},
      author={Buttazzo, Giuseppe},
      author={Michaille, G{\'e}rard},
       title={Variational analysis in {S}obolev and {BV} spaces: applications
  to {PDE}s and optimization},
   publisher={SIAM},
        date={2014},
}

\bib{CL131}{article}{
      author={Chen, Jingyi},
      author={Lamm, Tobias},
       title={A {Bernstein} type theorem for entire {Willmore} graphs},
        date={201301},
     journal={Journal of Geometric Analysis},
      volume={23},
       pages={456\ndash 469},
}

\bib{ChenLi2017}{article}{
      author={Chen, Jingyi},
      author={Li, Yuxiang},
       title={Radially symmetric solutions to the graphic {Willmore} surface
  eqaution},
        date={2017},
     journal={Journal of Geometric Analysis},
      volume={27},
       pages={671\ndash 688},
         url={https://doi.org/10.1007/s12220-016-9694-y},
}

\bib{chenli2022}{article}{
      author={Chen, Jingyi},
      author={Li, Yuxiang},
       title={Uniform convergence of metrics on {Alexandrov} surfaces with
  bounded integral curvature},
        date={2022-08},
     journal={arXiv e-prints},
       pages={arXiv:2208.05620},
      eprint={2208.05620},
}

\bib{conway2012functions}{book}{
      author={Conway, John~B},
       title={Functions of one complex variable {II}},
   publisher={Springer Science \& Business Media},
        date={2012},
      volume={159},
}

\bib{DD06}{article}{
      author={Dziuk, Gerhard},
      author={Deckelnick, Klaus},
       title={Error analysis of a finite element method for the {Willmore} flow
  of graphs},
        date={2006},
     journal={Interfaces Free Bound},
      volume={8},
       pages={21\ndash 46},
}

\bib{Fujimoto}{article}{
      author={Fujimoto, Hirotaka},
       title={On the number of exceptional values of the {Gauss} maps of
  minimal surfaces},
        date={1988},
     journal={Journal of the Mathematical Society of Japan},
      volume={40},
      number={2},
       pages={235 \ndash  247},
         url={https://doi.org/10.2969/jmsj/04020235},
}

\bib{helein2002harmonic}{book}{
      author={H{\'e}lein, Fr{\'e}d{\'e}ric},
       title={Harmonic maps, conservation laws and moving frames},
   publisher={Cambridge University Press},
        date={2002},
      number={150},
}

\bib{Hartman1959OnSI}{article}{
      author={Hartman, Philip},
      author={Nirenberg, Louis},
       title={On spherical image maps whose jacobians do not change sign},
        date={1959},
     journal={American Journal of Mathematics},
      volume={81},
       pages={901},
}

\bib{Li2016}{article}{
      author={Li, Yuxiang},
       title={Some remarks on {Willmore} surfaces embedded in $\mathbb{R}^3$},
        date={2016},
     journal={Journal of Geometric Analysis},
      volume={26},
       pages={2411\ndash 2424},
         url={https://doi.org/10.1007/s12220-015-9631-5},
}

\bib{LS14}{article}{
      author={Luo, Yong},
      author={Sun, Jun},
       title={Remarks on a {Bernstein} type theorem for entire {Willmore}
  graphs in $\mathbb{R}^3$},
        date={201407},
     journal={Journal of Geometric Analysis},
      volume={24},
       pages={1613\ndash 1618},
}

\bib{Li2019MetricsOA}{article}{
      author={Li, Yuxiang},
      author={Sun, Jianxin},
      author={Tang, Hongyan},
       title={Metrics on a surface with bounded total curvature},
        date={2019},
     journal={International Mathematics Research Notices},
}

\bib{Massey1962SurfacesOG}{article}{
      author={Massey, William~Schumacher},
       title={Surfaces of {Gaussian} curvature zero in {Euclidean} 3-space},
        date={1962},
     journal={Tohoku Mathematical Journal},
      volume={14},
       pages={73\ndash 79},
}

\bib{RT08}{article}{
      author={Rivière, Tristan},
       title={Analysis aspects of {W}illmore surfaces},
        date={200810},
     journal={Inventiones mathematicae},
      volume={174},
}

\end{biblist}
\end{bibdiv}
\end{document}